\documentclass{amsart}

\usepackage{amsmath, amssymb, amsfonts}
\usepackage{xy} \input xy \xyoption{all}

\newtheorem{thm}{Theorem}

\newtheorem{prop}{Proposition}
\newtheorem{remark}{Remark}
\theoremstyle{definition}

\def\bC{{\mathbb C}}

\def\Z{\mathbb Z}

\def\P{{\mathbb P}}
\def\L{{\mathcal L}}
\def\M{{\mathcal M}}

\def\H{\mathcal H}
\def\D{\Delta}
\def\Y{\mathcal Y}
\def\A{\mathcal A}
\def\Hs{\mathcal H_\s}
\def\l{{\lambda}}

\def\iso{{\, \cong\, }}
\def\<{\langle}
\def\>{\rangle}

\def\iso{\cong}
\def\f{\psi}
\def\ff{\phi}
\def\s{\sigma}
\def\ph{\varphi}

\begin{document}

\title{Genus 2 curves that admit a degree 5 map to an elliptic curve}

\keywords{genus 2 curves, Hurwitz spaces,  Humbert surfaces, covers of curves}

\subjclass[2000]{Primary: 14H10, 14H30,  Secondary: 14H45, 14H35}

\author{K. Magaard}
\address{Department of Mathematics, Wayne State University, Detroit,
MI 48201, USA}
\email{kaym@math.wayne.edu}
\thanks{The first author was partially supported by the NSA}

\author{T. Shaska}
\address{Department of Mathematics, Oakland University, Rochester, MI 48309, USA}
\email{shaska@oakland.edu}

\author{H. V\"olklein}
\address{Institut f\"ur Experimentelle Mathematik, Universit\"at Duisburg-Essen, Ellernstr. 29, D-45326 Essen, Germany  }
\email{voelkle@iem.uni-due.de}

\date{\today}

\begin{abstract}
We continue our study of genus 2 curves $C$ that admit a cover $ C \to E$ to a genus 1 curve $E$ of prime degree
$n$.  These curves $C$ form an  irreducible  2-dimensional subvariety $\L_n$ of the moduli space $\M_2$ of  genus
2 curves. Here we study the case $n=5$. This extends earlier work for degree 2 and 3, aimed at illuminating the
theory for general $n$.

We compute a normal form for the curves in the locus $\L_5$ and its three distinguished subloci. Further, we
compute the equation of the elliptic subcover in all cases, give a birational parametrization of the subloci
of $\L_5$ as subvarieties of $\M_2$ and classify all curves in these loci which have extra automorphisms.
\end{abstract}

\maketitle

\section{Introduction}

We continue our study of genus 2 curves $C$ that admit a cover $ C \to E$ to a genus 1 curve $E$. We assume
the degree $n$ of the cover is a  prime. For  $n=2, 3$ such curves $C$ already occur in the work of Hermite,
Goursat, Burkhardt, Brioschi, and Bolza, see Krazer \cite{Krazer} (p. 479). For general $n$, they have been
studied by Frey, Kani and others in  \cite{Fr, FK, FKV, Kani, Sh, Sh-V, Sh-F, Ku}. The following can be found
in these references. These curves $C$ are parametrized by an irreducible 2-dimensional subvariety $\L_n$ of
the moduli space $\M_2$ of genus 2 curves.  If $C$ corresponds to a generic point of $\L_n$ then $C$ admits
exactly two degree $n$ covers $ C \to E_1 $ and $ C \to E_2$ to a genus 1 curve, up to equivalence. Here we
call two such covers equivalent if they correspond to the same elliptic subfield of the function field of
$C$. The Jacobian of $C$ is isogenous to $E_1 \times E_2$; see \cite{Sh1} for details.

The degree $n$ cover $\f: C \to E$ induces a degree $n$ cover $\ff: \P^1 \to \P^1$ such that the following
diagram commutes

\begin{figure}[hb!]
$$
\entrymodifiers={+[o]} \SelectTips{cm}{}
\xymatrix{
C  \ar[r]^{\pi_C} \ar[d]_{\psi} & \P^1 \ar[d]^{\phi}\\
E \ar[r]^{\pi_E} & \P^1 \\
}
$$
\caption{The  basic setup}\label{Diag1}
\end{figure}

Here $\pi_C: C \to \P^1$ and $\pi_E: E \to \P^1$ are the natural degree 2 covers.  Let $r$ be the number of
branch points of the cover $\ff: \P^1 \to \P^1$. Then $r=4$ or $r=5$, with $r=5$ being the generic case and
$r=4$ occurring for a certain 1-dimensional sub-locus of $\L_n$.  We refer to the case $r=5$ (resp., $r=4$)
as the \textbf{non-degenerate case}, resp., the \textbf{degenerate case}; see \cite{FKV} or Theorem 3.1 in
\cite{Sh1}.

Here we study the case $n=5$. This extends earlier work for  $n=2, 3$ in Shaska \cite{Sh1}, \cite{Sh-F}, and
Shaska/V\"olklein \cite{Sh-V}. So from now on we assume $n=5$. Then the cover $\ff: \P^1 \to \P^1$ has one of
the following ramification structures:

\medskip

\begin{table}[ht!]
\begin{tabular}{ccc}
\textit{non-degenerate:} &  &   $\left( \, (2)^2, (2)^2, (2)^2, (2), (2) \, \right)$ \\
\ \  \textit{degenerate:}     &  & \\
  & I)    &     $ \left(\, (2)^2, (2)^2, (4), (2) \, \right) $  \\
  & II)   &   \qquad   $ \left(\, (2)^2, (2)^2, (2)\cdot (3), (2) \,   \right)$   \\
  &  III) &    \ \ $ \left(\, (2)^2, (2)^2, (2)^2, (3) \, \right)$ \\ \\
\end{tabular}
\caption{ramification structure of $\ff$}
\end{table}
\noindent see \cite{FKV} or \cite{Sh1} for ramification structures of arbitrary degree. This data lists the
ramification indices $>1$ over the branch points. E.g., in the last case there is one branch point that has
exactly one ramified point over it, of index 3, and each of the other 3 branch points has exactly two
ramified points over it, of index 2.

The main feature that distinguishes the case $n=5$ from all other values  $n>5$ is that  the cover $\phi$
does not determine $\psi : C \to E$ uniquely, but there is essentially \textbf{two} choices of $\psi$ for a
given $\phi$. These two choices correspond to the two branch points of $\phi$ of ramification structure $(2)$
(notation as in Table 1) -- anyone of these two branch points can be chosen to ramify in $E$  while the other
doesn't. This phenomenon implies that the function field of $\L_5$ is a quadratic extension of the function
field of the Hurwitz space parameterizing the covers $\phi$.

We show that each of the 3  degenerate cases corresponds to an irreducible 1-dimensional locus  on $\L_5$,
two of which have genus zero and one has genus 1. We compute a normal form for the curves in the locus $\L_5$
and its three distinguished sub-loci.  We give a bi-rational parametrization of these sub-loci  as
subvarieties of $\M_2$ and  classify all curves in them that have extra automorphisms.

\medskip

\noindent \textbf{A few remarks on computations:} The computations of this paper were performed using Maple
or Mathematica. While we provide a sketch of such computations, we skip most of the details. We are assuming
that the reader is familiar with some computational algebra packages and knows the basic methods of "solving" systems of non-linear equations (i.e., Groebener bases algorithms, resultants, etc). Throughout the paper we
use the terms "computations are easy" or "straightforward". This should not be confused with "fast" or
"quick". Some of these computations took several days. For the interested reader who wants to re-produce such results we provide details on \cite{Eq}. The explicit equations of the loci $\Y_i$, $i=1,2,3$ or the list of
genus two  curves with extra automorphisms (cf. Section 4) which are in the locus $\Y_i$ can be provided by
the second author upon request.

\section{Background on Hurwitz spaces and Humbert surfaces}

\subsection{Hurwitz spaces of covers $\phi : \P^1 \to \P^1$}
Two covers $f:X\to\P$ and $f':X'\to\P$ are called   \textbf{weakly equivalent} if there is a homeomorphism
$h:X\to X'$ and an analytic automorphism $g$ of $\P$ (i.e., a Moebius transformation) such that  $g\circ
f=f'\circ h$. The covers $f$ and $f^\prime$ are called \textbf{  equivalent } if the above holds with $g=1$.

Consider a cover $f:X \to \P$ of degree $n$, with branch points $p_1,...,p_r\in\P$. Pick $p\in
\P\setminus\{p_1,...,p_r\}$, and choose loops $\gamma_i$ around $p_i$ such that $\gamma_1,...,\gamma_r$ is a
standard generating system of the fundamental group $\Gamma:=\pi_1( \P\setminus\{p_1,...,p_r\},p)$ (see
\cite{Buch}, Thm. 4.27); in particular, we have $\gamma_1\cdots\gamma_r=1$. Such a system
$\gamma_1,...,\gamma_r$ is called a homotopy basis of $\P\setminus\{p_1,...,p_r\}$. The group $\Gamma$ acts
on the fiber $f^{-1}(p)$ by path lifting, inducing a transitive subgroup $G$ of the symmetric group $S_n$
(determined by $f$ up to conjugacy in $S_n$). It is called the \textbf{monodromy group} of $f$. The images of
$\gamma_1,...,\gamma_r$ in $S_n$ form a tuple of permutations $\s=(\s_1,...,\s_r)$ called a tuple of \textbf{
branch cycles} of $f$.

We say a cover $f:X\to\P$ of degree $n$ is of type $\s$ if it has $\s$ as tuple of branch cycles relative to
some homotopy basis of $\P$ minus the branch points of $f$. Let $\Hs$ be the set of weak equivalence classes
of covers of type $\s$.  The \textbf{Hurwitz space} $\Hs$ carries a natural structure of an quasiprojective
variety (see \cite{FrV},\cite{Buch}).

We have $\Hs=\H_\tau$ if and only if the tuples $\s$, $\tau$ are in the same    \textbf{braid orbit}
$\mathcal O_\tau = \mathcal O_\sigma$  (see \cite{Buch}, Def. 9.3 or \cite{Ma}). In the case of the covers
$\phi : \P^1 \to \P^1$ from above, the corresponding braid orbit consists of all tuples in $S_5$ whose cycle
type matches the ramification structure of $\phi$.  This and   the genus of $\Hs$ in the degenerate cases
(see the following table) has been computed by the BRAID PACKAGE; see \cite{Ma}.

\begin{table}[ht!]
\begin{tabular}{c|c|c|c|c|c }
Case &   cycle type of $\s$ & $\# ( \mathcal O_\s )$     & $G$ & $\dim \Hs$ & genus of $\Hs$ \\
\hline &      && &   &   \\
& $(2^2, 2^2, 2^2, 2, 2) $  &  40   & $S_5 $   & 2 & -- \\
&      & &    & & \\
I) & $(2^2, 2^2, 4, 2) $    &  8   &$S_5 $   & 1  & 0 \\
II)& $(2^2, 2^2, 2\cdot 3, 2)$& 6  & $S_5$    & 1  & 0 \\
III) & $(2^2, 2^2, 2^2, 3)$  &  9  &  $A_5$    & 1  & 1 \\
\end{tabular}
\vspace{0.6cm}
\caption{Hurwitz spaces and their dimensions}
\end{table}
In the next section we compute a normal form for the curve $C$ in $\L_5$ and each of its degenerates subloci.
Further, we find birational parametrizations for each degenerate sublocus of $\L_5$. The non-degenerate case is a
little more complicated computationally; we omit the details.

\subsection{Humbert surfaces}
Let $\A_2$ denote the moduli space of principally polarized abelian surfaces. It is well known that $\A_2$ is the
quotient of the Siegel upper half space $\mathfrak H_2$ of symmetric complex $2 \times 2$ matrices with positive
definite  imaginary part by the action of the symplectic group $Sp_4 (\Z)$; see \cite{G} (p. 211).

Let $\D$ be a fixed positive integer and  $N_\D$ be the set of matrices
\[\tau =
\begin{pmatrix}z_1 & z_2\\
z_2 & z_3
\end{pmatrix}
\in \mathfrak H_2\]
such that there exist nonzero integers $a, b, c, d, e $ with the following properties:
\begin{equation}\label{humb}
\begin{split}
& a z_1 + bz_2 + c z_3 + d( z_2^2 - z_1 z_3) + e = 0\\
& \D= b^2 - 4ac - 4de\\
\end{split}
\end{equation}

The {\it Humbert surface}  $\H_\D$  of discriminant $\D$   is called the image of $N_\D$ under the canonical map
\[\mathfrak H_2 \to \A_2:= Sp_4( \Z)\setminus{\mathfrak H}_2,\]
see \cite{Hu}, \cite{BW} or \cite{Mu} for details.  It is known that $\H_\D \neq \emptyset$ if and only if $\D
> 0$ and $\Delta \equiv 0 \textit { or } 1 \mod 4$. Humbert (1900) studied the zero loci in
Eq.~\eqref{humb} and discovered certain relations between points in these spaces and certain plane configurations
of six lines; see \cite{Hu} for more details.

For a genus 2 curve $C$ defined over $\bC$, $[C]$ belongs to $\L_n$ if and only if the isomorphism class
$[J_C] \in \A_2$ of its (principally polarized) Jacobian $J_C$ belongs to the Humbert surface $\H_{n^2}$,
viewed as a subset of the moduli space $\A_2$ of principally polarized abelian surfaces; see \cite{Mu}
(Theorem 1, pg. 125) for the proof of this statement. In particular, every point in $\H_{n^2}$ can be
represented by an element of $\mathfrak H_2$ of the form
\[\tau
=
\begin{pmatrix}z_1 & \frac 1 n \\
\frac 1 n & z_2
\end{pmatrix}, \qquad z_1, \, z_2 \in \mathfrak H.
\]

There have been many attempts to explicitly describe these Humbert surfaces. For some small discriminant this
has been done by several authors; see \cite{Sh-V}, \cite{Sh-F}, \cite{Ku}. Geometric characterizations of
such spaces for $\D= 4, 8, 9$, and 12 were given by Humbert (1900) in \cite{Hu} and for $\D= 13, 16, 17, 20$,
21 by Birkenhake/Wilhelm (2003) in \cite{BW}. The Humbert surface of discriminant 25 has not been explicitly
described.

\section{Parametrization of the covers $\phi$}
In Theorem~\ref{thm1}, ii)  we give an explicit equation for the cover $\phi: \P^1 \to \P^1$ associated with
any degree 5 cover $C\to E$  (as in the Introduction). It is easier to formulate the result in terms of
function fields as follows:
\[
\xymatrix{
K_2 \ \  \ar@{-}[d]   \ar@{-}[r]   & \ \  \bC(x) \ar@{-}[d] \\
K_1 \ \      \ar@{-}[r]        & \ \ \bC (t)}
\]
\begin{thm}\label{thm1}
Let $K_2$ be a genus 2 function field over $\bC$ and $K_1$ a genus 1 subfield with  $[K_2: K_1]=5$. Then there are
$x, t \in K_2$, unique up to the action of $S_3$ given by the transformations
\begin{equation}\label{S3_x_t}
(x, t)\ \ \mapsto \ \ (\frac 1 x, \frac 1 t), \ \ \ \ (x, t) \ \ \mapsto \ \ ( 1 - x,  1 - t),
\end{equation}
and $a, b \in \bC\setminus \{ 0\}$, $a+b \neq -1$, unique up to the action of $S_3$ given by the transformations
\begin{equation}\label{S3_a_b}
(a, b) \ \ \mapsto \ \ ( \frac a b, \frac 1 b), \ \ \ \ (a, b) \ \ \mapsto \ \ (a, - a -b-1)
\end{equation}
such that the following holds:\\

i)   $\bC (x)$ is the unique  rational subfield of $K_2$ of degree 2. The generator $\iota$ of $Gal
(K_2/\bC(x))$ is called the hyperelliptic involution of $K_2$ and it fixes $K_1$.\\

ii) $\bC(t)= K_1 \cap \bC (x)$ and
$$t \ = \ \ph (x) \ = \ \  x \left(\frac {F_1 (x)} {F_2(x)}\right)^2$$
\begin{equation}
t -1  \  \ = \  \  (x-1)\left(\frac {F_3 (x)} {F_2(x)}\right)^2
\end{equation}
where
\begin{small}
\begin{equation}
\begin{split}
F_1(x) & =x^2+(2a+2b+a^2)\,x+2ab+b^2 \\
F_2 (x)& =(2a+1)\,x^2+(a^2+2ab+2b)\,x+b^2\\
F_3 (x)& = x^2 - (a^2-2b)\,x+b^2\\
\end{split}
\end{equation}
\end{small}
\end{thm}

\begin{proof}  i) is well-known (see  e.g. \cite{FKV}).

Since $\iota$ fixes $K_1$, the field $K_1 \cap \bC (x)$ is a subfield of $K_1$ of degree 2. It is of the form
$\bC(t)$, where $t= \ph (x)$ is a rational function of $x$ of the ramification structure described in Table
1. We normalize $t$  by assuming that $\ph$ is ramified at $t=0, 1,  \infty$ of type not equal to $(2)$ or
$(3)$ (notation as in Table 1). Similarly we normalize $x$  by assuming that $x=0, 1, \infty$ is that place
over $t=0, 1, \infty$, respectively, that is unramified except in the case that there is no unramified place
over the corresponding value of $t$ in which case we assume that it has ramification index 3 (The latter
occurs only in the degenerate case II). From Table 1 we see that this normalization determines $x$ and $t$
uniquely up to the action of the subgroup of $PGL_2 (\bC)$ permuting $0, 1, \infty$. Indeed, the three places
that we assign the values $0, 1, \infty$ are determined uniquely by $K_1$ and $K_2$, but we have to order the
three places by assigning these values and all the possible orderings are conjugate under the group $S_3$
from above. This proves the uniqueness assertion on $x, t$.

From the ramification structure of $\ph$ and the above normalisations it follows  that $\ph(X)/X$ and
$(\ph(X)-1)/(X-1)$ are squares in $\bC (X)$ (where $X$ is a variable). Thus
$$ \ph(X)  = X \frac {(X^2 + MX +N)^2} {(AX^2 + BX + C)^2}$$
and the condition $\ph (1) =1$ implies that $1+M+N= \pm (A+B+C) $.  Replacing $A, B, C$ by their negatives if
necessary we can in any case write $\ph (X)$ in the form
$$\ph(X)  = X \frac {(X^2 + MX +A+B+C - M -1)^2} {(AX^2 + BX + C)^2}.$$
Factoring $\ph (X)-1$ now yields an expression of the form
$$\ph (X)-1 = (X-1) \frac {g(X)} {(AX^2 + BX + C)^2}$$
where $g(X)$ is a monic polynomial of degree 4. We know that $g(X)$ has to be a square. To exploit this condition,
we observe that there is a simple criterion for a monic degree 4 polynomial to be a square, see Remark~\ref{rem1}
below. Applying this criterion to $g(X)$ yields
\begin{equation}\label{eq1}
\begin{split}
& 2M-8 A-8 B-8 C+8 A B+16 A C+11 A^2+8 B^2+4 M A^2+8 M^2 A^2-\\
& 6 M A^4-8 A^2 C  +8 A^3 B-8 A^2 B-16 M A B+A^4-8 A^3+A^6+3=0\\ \\
\end{split}
\end{equation}
and
\begin{equation}\label{eq2}
\begin{split}
&(A-1) (-A^3-A^2+4 A M-3 A-8 B+5+4 M) (-4 M+4 M A^2-A^4-2 A^2\\
& -8 A B+8 A+8 B+16 C-5)=0\\
\end{split}
\end{equation}

If $A=1$, then the equation Eq.~\eqref{eq1}  reduces to $B=M$, which implies that $\ph (X)=X$: a
contradiction. If the last factor in  equation Eq.~\eqref{eq2} vanishes, then from  Eq.~\eqref{eq1}  we
compute that the resultant of $X^2+MX+N$ and $AX^2 + BX + C$ vanishes. Hence $\ph (X)$ has  degree $\leq 3$,
again a contradiction. Therefore, Eq.~\eqref{eq2} reduces to
$$-A^3-A^2+4 A M-3 A-8 B+5+4 M=0.$$
From this we get
$$B=\  \frac 1 8 \ \ (-A^3-A^2-3A+4AM+4M+5).$$
Plugging this into Eq.~\eqref{eq1} we get
$$C =\  \frac 1 { 64} \ \ (A^2  +2A -4M -3)^2$$
Now we define $a, b$ as follows
$$a = \ \frac {A-1} 2, \ \ \ \ b= \ \frac 1 8 (-A^2+6A-4M-5)$$
One can easily check that the formulas in ii) hold.  Furthermore, $a \neq 0$ because $A\neq 1$ as noted
before. Also, $b\neq 0$ and  $a+b \neq -1$  because otherwise the resultant of $F_1(X)$ and $F_2(X)$ vanishes
(which is a contradiction because then $\deg \ph (X) \leq 3$).

It remains to prove the uniqueness assertion on $a, b$. Clearly, $a, b$ are uniquely determined by $\ph (X)$,
hence by  $x, t$. Therefore $a, b$ are unique up to the action of $S_3$ as in Eq.~\eqref{S3_x_t}. The induced
action on $\ph (X)$ is generated  by the transformations
\begin{equation}\label{action}
\ph (X) \mapsto \frac 1 {\ph (1/X)} \ \ \    and   \ \ \  \ph (X) \mapsto  1 - {\ph (1-X)}.
\end{equation}
From this we compute that $S_3$ acts  on $a, b$ as in Eq.~\eqref{S3_a_b}.

\end{proof}

\begin{remark}\label{rem1}
The polynomial
$$X^4 + \alpha X^3 + \beta X^2 + \gamma X + \delta $$
is a square in $\bC (X)$ if and only if
\[8 \ \gamma = \alpha (4 \beta - \alpha^2) \ \ \ \mbox{and} \ \ \ 64\ \delta=(4 \beta - \alpha^2)^2\]
\end{remark}
The proof is by direct computation.

\begin{remark}\label{rem2} (Significance of the $S_3$ action)\\
Let  $\ph (X)$  as in  Theorem~\ref{thm1} and let $\ph_1 (X)$ be another function of the same shape. Then the
corresponding covers $\P^1 \to \P^1$  are equivalent if and only if $\ph = \ph_1$. And the covers are weakly
equivalent if and only if $\ph $ and $\ph_1$ are conjugate under the $S_3$ action  from Eq.~\eqref{action}.
The induced $S_3$ action on the parameters $a$, $b$ is given in Theorem~\ref{thm1}. We compute the fixed
field of $S_3$ in $\bC (a, b)$ to be $\bC (u, v)$ where
$$u = \frac {2 a \ (a b+b^2+b+a+1)} {b \ (a+b+1)},  \qquad v= \frac {a^3} { b \ (a+b+1)}.$$
Thus the Hurwitz space classifying (up to weak equivalence) the covers $\phi $ of non-degenerate type is
birationally parameterized by $u$ and $v$. And the function field of $\L_5$ is a quadratic extension of $\bC
(u, v)$, as we will see later.
\end{remark}

\begin{thm}\label{thm2}
In the situation of Theorem~\ref{thm1}, normalize $x, t$ further by assuming that if one of the degenerate
cases I) or II) of Table 1 occurs then the exceptional  ramification of $\bC (x) / \bC(t)$ occurs over $t=1$.
Then the polynomial
$$F_4 (X)\  =  \ (2 a+1) \,X^2+(2 b-2 b a-2 a-a^2) \,X+b^2+2ab$$
has a root $z \neq 1$ and the genus 2 field $K_2$ is generated by $x$ together with another element $y$ satisfying
\begin{equation}\label{curve}
y^2 \ = \  x (x-1) \ g_3 (x),
\end{equation}
where $g_3 (x)$ is given in Eq.~\eqref{g3} below. The polynomial $g_3 (x)$ has coefficients in $\bC [a, b,
z]$.
\end{thm}

\noindent Here are the coefficients of $g_3 (x):= a_3 x^3 + a_2 x^2 + a_1 x + a_0$:

\begin{small}
\begin{equation}\label{g3}
\begin{split}
a_0 = \, & -b^4 (2 b^3 a+4 b^3-2 z a b^2+7 b^2 a^2+8 z b^2+4 b^2 +16 a
b^2+16 z b a+6 a^3 b+8 b a\\
& +2 z a^2 b+12 z b+16 b a^2+13 z a^2+z a^4+6 z a^3+4 z+12 y a)\\
a_1 = \, & -b^2 (12 b^3+12 b^4 a+32 z b a-6 a^4 b^2+44 b^2 a^3+6 b a^2+24 a
b^2+10 a^3 b+44 b^3 a^2+2 b a \\
& +52 b^3 a+61 b^2 a^2-12 b a^5-7 z a^2-2 z a+12 z b-4 a^6+12 b^4-a^4-40 z
a^3 b^2-16 z b^3 a^2 \\
& -12z a^5+36 z b^2-18 z a^3-26 z a^4+56 z a b^2+4 a z b^3+2 z a^2 b^2-20 z
a^3 b +28 z a^2 b\\
& +2 z a^6+24 z b^3+4 z b a^5-4 a^5-32 z a^4 b)\\
a_2 = \, &  5b^2 a^6+20 b^2 a^5+8 b a^6-61 b^4 a^2-18 b^5 a-56 b^4 a+4 z b
a+5 a^4 b^2-18 b^2 a^3 -24 z b^4\\
& -14z b^4 a-4 a b^2+8 b^3 a^4+2 b^3 a^5-54 b^3 a^3-70 b^3 a^2-24 b^3 a-14
b^2 a^2+4 a^4 b+10 b a^5 \\
& -6 z a^7+64 z a^3 b^3+38 z a^4 b^2+54 z a^3 b^2+12 z b^3 a^2-14 z a^6 b-10
z b^2 a^5-4 z a^7 b-4 a^6 z b^2 \\
& +32a^2 b^4 z+2 a^7 b-z a^8-36 z b^3-12 z a^5-12 z b^2-4 z a^4-28 z a
b^2-64 a z b^3-5 z a^2 b^2 \\
& +16 z a^2 b+28 z a^4 b-4 z b a^5-13 z a^6-12 b^5-12 b^4+34 z a^3 b\\
a_3  =\, & (2 a+1) (z a^4-2 a^3 b+4 z a^3+6 z a^3 b-4 b a^2+12 z a^2 b^2+10
z a^2 b-9 b^2 a^2+5 z a^2 \\
& -2 b a+2 z a-8 a b^2-12 b^3 a+8 a z b^3-4 b^3-4 z b-4 b^4-12 z b^2-8 z
b^3)\\
\end{split}
\end{equation}
\end{small}

\begin{proof}The derivative of $\ph (X) $ factors as follows:
\begin{small}
\begin{equation}
\begin{split}
\ph^\prime (X)  = \, \frac {F_1(X) \cdot F_3(X) \cdot F_4 (X)} {F_2^3(X)}
\end{split}
\end{equation}
\end{small}
If $F_4 (X) $ has no root $z\neq 1$ then $\bC (x) / \bC (t)$ is unramified outside $t=0, 1, \infty$. This
contradicts Table 1. Thus there is a root $z \neq 1$ of $F_4 (X)$. Furthermore, $\deg F_4 (X)=2$, because if $2a
+1 =0$ then $\deg F_2 (X) =1$ and degenerate case II) from Table 1 occurs with exceptional ramification over
$t=\infty$. But this is excluded in the theorem.

The numerator of  $ {\ph(X)-\ph(z)} $ is a polynomial $G(X, z)$ of degree 5  in $X$ as well as in $z$. As a
polynomial in $X$ it has $X=z$ as a double root because $\ph^\prime (z)=0$. Using the equation $F_4(z)=0$ we can
re-write every polynomial in $X$ and $z$ such that it becomes linear in $z$. Thus
\begin{equation}\label{G}
G( X, z)  = (X-z)^2 \ \cdot \ (A(x) z \ +  \ B(x))
\end{equation}
To find $A(x)$ and $B(x)$ we re-write both sides of Eq.~\eqref{G} so that they become linear in $z$, and then
we compare the $z^1$-coefficient and the $z^0$-coefficient on both sides. The result is displayed in
Eq.~\eqref{g3}. It follows from \cite{FKV},  that $K_2 = \bC (x, y)$ with $x$ and $y$ satisfying
Eq.~\eqref{curve}.

\end{proof}

\begin{remark}
In the situation of Theorem~\ref{thm1}, assume that none of  the degenerate cases I) or II) of Table~1
occurs. Then,
\begin{equation}\label{disc}
\begin{split}
\Delta (a, b) \ = \ (a+b+1) b (2 a+1) (a-b-1) (a^2-4 b) (4 b+4+4 a+a^2) \\
(4 b^2+4 b+4 b a+a^2)  (a^3-2 b-2 b a-2 b^2) (2 a+b) & \neq 0
\end{split}
\end{equation}
\end{remark}

The above assumption implies  $F_4 (1) \neq 0$, thus $z$ from Theorem~\ref{thm2} can be any root of $F_4
(X)$. The discriminant of the right hand side of the Eq.~\eqref{curve} is non-zero. We compute that
discriminant as a polynomial in $a, b, z$. Taking the resultant of this polynomial and $F_4 (z)$ with respect
to $z$ yields $\Delta$ as in Eq.~\eqref{disc}. Therefore, $\Delta \neq 0$.

\begin{remark}
It follows from  Theorem~\ref{thm2} that the function field of locus $\L_5$ is contained in the field $\bC
(a, b, z)$, where $F_4 (z)=0$ (and $a, b$ are algebraically independent). The $S_3$ action from
Theorem~\ref{thm1} extends naturally to the field $\bC (a, b, z)$ because the action of $S_3$ is
$\ph$-equivariant. It is generated by the transformations
\begin{equation}\label{S3_a_b_z}
\sigma :  (a, b, z)\ \mapsto \ \ ( \frac a b, \frac 1 b, \frac 1 z), \ \ \ \tau :  (a, b, z) \ \ \mapsto \ \
(a, - a -b-1, 1-z).
\end{equation}
\end{remark}

Let $H=\< \s, \tau\>$. Then, $H \iso S_3$. We have the following:

\begin{thm}\label{thm3}
The function field of $\L_5$ is given by $\bC (\L_5)= \bC(a, b, z)^H$. Moreover, the invariants of such
action are $u, v $ and $w$ where
\[w = \frac {(z^2-z+1)^3} {z^2 (z-1)^2}\]
In particular,
\[\bC (\L_5) = \bC (u, v, w), \]
where the equation of $w$ in terms of $u, v$ is
\[ c_2 w^2 + c_1 w + c_0 =0\]
with $c_0, c_1, c_2$ as follows:
\begin{equation}\label{eq_w}
\begin{split}
c_2  = & \,  64v^2(u-4v+1)^2\\
c_1  = & \, -4v(-272v^2u-20 v u^2+2592 v^3-4672 v^2+4 u^3+16 v^3 u^2-15 v u^4\\
&\,  -96 v^2 u^2+24 v^2 u^3+2 u^5-12 u^4+92 v u^3+576 v u-128 v^4-288 v^3 u)\\
c_1  = & \, (u^2+4 v u+4 v^2-48 v)^3 \\
\end{split}
\end{equation}

\end{thm}

\begin{proof}From definitions of $u$ and $v$ we determine that $[\bC (a, b) : \bC (u, v)]=6$. Then, $[\bC (a, b, z):
\bC(u,v)]=12$. The irreducible polynomial of degree 12 of $z$ over $\bC (u, v)$ can be easily determined; see
\cite{Eq} for details. Hence, $[ \bC (a, b, z)^H  : \bC (u, v)]=2$. It is easily verified that $u, v$ and $w$
are invariants under the action of $H$. Hence, $\bC(u,v,w)$ is a subfield of $\bC (a,b, z)^H$. It is left to
show that $[\bC(u,v,w): \bC (u, v)]=2$, see Fig.~\ref{fig2}. We have the system of equations

\begin{figure}[ht!]
$$
\entrymodifiers={+[o]} \SelectTips{cm}{} \xymatrix{
&   &  \, \,    {   \bC (a, b, z) } \ar@{-}[d]^{S_3}  \ar@{-}[dl]^{2}    &   \\
& \bC (a, b) \ar@{-}[d]^{S_3}  &  \, \,       \bC(u, v, w) =\bC(\L_5)  \ar@{-}[dl]^2      &   \\
& \bC(u,v) =\bC( \Hs) &      &    \\
}
$$
\caption{The  function field of $\L_5$}  \label{fig2}
\end{figure}
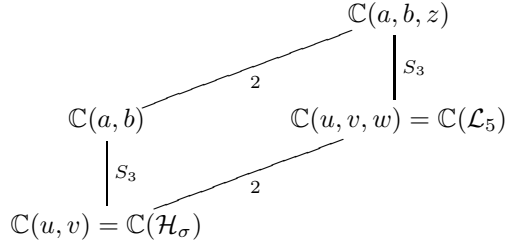

\begin{equation}\label{sys}
\left\{
\begin{aligned}
& u b (a+b+1) - 2 a (a b+b^2+b+a+1)=0\\
& v b (a+b+1) - a^3=0\\
& w z^2 (z-1)^2 - (z^2-z+1)^3=0 \\
& (-2a-1)z^2+(-2b+2a+2ab+a^2)z-b^2-2ab=0\\
\end{aligned}
\right.
\end{equation}
Since $\bC (a, b) \subset \bC (u, v, z)$   we can express $a$ and $b$ as rational functions in $u, v, z$; see
\cite{Eq} for explicit expressions. From the above system we are left with the third equation and the degree
12 polynomial in $z$ with coefficients in $\bC (u, v)$. Taking the resultant of these two polynomials with
respect to $z$ gives  a degree 2 irreducible polynomial of $w$ with coefficients as in Eq.~\eqref{eq_w}. This
completes the proof.

\end{proof}

%
%
%
%
%
%

\subsection{Computing the locus $\L_5$}
We denote by $J_2, J_4, J_6, J_{10}$ the classical invariants of $C$, for
their definitions see \cite{Ig} or \cite{Vishi}.  These are homogeneous polynomials (of degree indicated by
subscript) in the coefficients of a sextic $f(X,Z)$ defining $C$
$$Y^2=f(X,Z)=a_6X^6+ a_5X^5Z + \dots + a_1XZ^5+a_0$$ and they are a
complete set of $SL_2(k)$-invariants (acting by coordinate change). They yield homogeneous coordinates on the
moduli space $\M_2$ (Igusa coordinates). The corresponding inhomogeneous coordinates are the absolute invariants
\begin{small}
\begin{equation}
i_1:=144 \frac {J_4} {J_2^2}, \quad i_2:=- 1728 \frac {J_2J_4-3J_6} {J_2^3}, \quad i_3 :=486 \frac {J_{10}}
{J_2^5}
\end{equation}
\end{small}
Two genus 2 curves with $J_2\neq 0$ are isomorphic if and only if they have the same absolute invariants.

For a curve in  $\L_5$ we can express $i_1, i_2,i_3 $ in terms of $a, b, z$ by using Eq.~\eqref{curve}. In
the degenerate cases they only depend on one parameter which we eliminate to obtain an equation in $i_1, i_2,
i_3$ defining the corresponding locus in $\M_2$; see next section for details.

Since from Eq.~\eqref{sys} we can express $a, b$ as rational functions in $u, v, z$, then $i_1, i_2, i_3$ are
given as rational functions in $u, v, z$. By using the definition of $w $ in terms of $z$ we express $i_1,
i_2, i_3$ in terms of $u, v$, and $w$. From the equation of $w$ in terms of $u, v$ (this is a degree 2
polynomial in $w$ with coefficients in $\bC (u, v)$), we  eliminate $w$ and are left with three equations
\[
 f_1 (i_1, u, v)=0, \quad  f_2 (i_2, u, v)=0, \quad  f_3 (i_3, u, v)=0.  \]
Eliminating $u$ and $v$ gives the equation of $\L_5$. This equation is available at \cite{Eq}.

\section{Degenerate cases}

Let $\Y_1$ (resp., $\Y_2$, resp., $\Y_3$)  be the locus in $\M_2$ of the genus 2 curves admitting a degree 5
cover $ C \to E$ such that the corresponding cover $\phi: \P^1 \to \P^1$   is of degenerate type I (resp.,
II, resp., III). In these cases the cover $\phi$ has only four branch points all of which are ramified in $E$
(in the situation of Fig.~\eqref{Diag1}). From their ramification structure it follows that in these cases
there is a 1-1 correspondence between $\phi$ and $\psi$. In other words,  each of the three curves $\Y_i$ is
isomorphic to the corresponding Hurwitz space. In order to compute that Hurwitz space, we take the equation
for $\phi$ from Theorem~\ref{thm1} and symmetrize by the $S_3$ action. In cases I and II, the branch point
$t=1$ is distinguished from $t=0$ and $t=\infty$ by its ramification structure. Thus, only the action of
$\sigma \in S_3$ which comes from permuting $0$ and $\infty$ has to be considered. That's why cases I and II
are easier than case III.  We treat them in the present section.

Let $\ph (x)$ be as in Theorem~\ref{thm1}.   The corresponding cover $\phi$ belongs to degenerate case  I
(resp., II), normalized as in Theorem~\ref{thm2}, if and only if  the following condition \textbf{  I}
(resp., \textbf{ II}) holds.
\[ 4 b=a^2 \leqno{(\textbf{  I})}\]
\[b = a-1 \leqno{(\textbf{  II})}\]
\subsection{Case I}
For case \textbf{I}, this follows because $\phi^{-1} (1)$ contains a point of ramification index 4 if and
only if the two roots of $F_3(x)$ collapse. The latter is equivalent to $\ Disc (F_3, x) =0$, which yields $b
= \frac {a^2} 4$. The argument in case \textbf{II} is similar (replacing the condition $\ Disc (F_3, x) =0$
by $F_3 (1)=0$).

We compute that the root $z$  of $F_4(x)$ not over $t=1$ equals
$$  \frac {a (8+a)} {4 (2 a+1)}, $$  respectively
$$ \frac {(3a-1)(a-1)}{(2a+1)}$$ in case I (resp., II).
Thus, from Eq.~\eqref{curve} and Eq.~\eqref{g3} we get that $C$ has  an equation of the form
\[y^2=x(x-1) (b_3 x^3+ b_2 x^2 + b_1 x + b_0), \]
where $b_0, \dots b_3$  depend only on $a$.

\begin{remark}
In case I the  branch points of $\phi$ are $t=0,1,\infty$ and
$$\l\ = \  \  \ph(z) \ \ = \ \   \frac  {4 (2 a+1)^3 (a^2+4
a+8)^2} {(2-a)^5 (a+2)^3}$$ Thus, if we fix $\l$  then there are 8 corresponding maps $\phi$ up to equivalence.
This is consistent with the fact that modulo conjugation there are exactly 8 generating systems $\s$ of $S_5$ of
cycle type $ ((2)^2, (2)^2, (4), (2))  $ and with product 1. There is a similar correspondence in case II.
\end{remark}

The absolute invariants of $C$ are rational functions in $a$ of degree 24, 36, and 60 respectively. We make the
substitution
$$T:=\left( \frac {a-2} {5 (a+2)} \right)^2 . $$
Then $i_1, i_2, i_3$ can be expressed as rational functions in $T$ of degree 12, 18, 30 as follows:

\begin{small}
\begin{equation*}
\begin{split}
i_1  & = \frac {45} {J_2^2}\, \left( 200225830078125 T^{12} -1719272460937500
T^{11}+565236035156250 T^{10} \right. \\
& -54100617187500 T^9+13999178671875 T^8 -4261746675000 T^7+606825435500 T^6\\
&\left. -54844543800 T^5+4205965699 T^4-236021164 T^3+6405914 T^2+6116 T-211) \right) \\ \\
i_2 & = \frac {135} {J_2^3}\,
\left( 6335270404815673828125T^{18}+113021224021911621093750T^{17} \right. \\
& -137079483776092529296875 T^{16}+35382386975097656250000T^{15} \\
& -5727170209350585937500 T^{14} +1661335117119140625000 T^{13}\\
& -438672743956054687500T^{12}+71535083209593750000 T^{11}\\
& -9593401735688906250T^{10} +1451100945145362500 T^9 -198805994903162250 T^8\\
&   +18781404045085680 T^7  -1082976623440908 T^6 +34245258932328 T^5 \\
& \left. -572847931740 T^4  +10845126800 T^3  -380189355 T^2+11646582T-3107 \right) \\ \\
i_3 & = \frac {49766400} {J_2^5} \, T \, (9 T-1)^5 \, (25 T^2+6 T+1)^5 \,
(25 T-1)^7 \\
\end{split}
\end{equation*}
\end{small}
\noindent where
\begin{small}
\begin{equation*}
J_2=5859375 T^6-129843750 T^5-31959375 T^4-6330100 T^3-54927 T^2+12506 T-17
\end{equation*}
\end{small}

\begin{prop}
${\mathbb C}(\Y_1) = {\mathbb C} (T)$
\end{prop}

\begin{proof}Computationally, one can solve $T$ from above equations and express it as a rational function in $i_1, i_2,
i_3$.

\end{proof}

Thus, the map
\begin{equation}
\begin{split}
\Phi: \quad &  {\mathbb C} \setminus \{ \D_1 \neq 0\}   \to \Y_1 \\
\qquad T & \to (i_1, i_2, i_3)\\
\end{split}
\end{equation}
gives a birational parametrization of $\Y_1$. The degrees of the field extensions are as in the following diagram.
\begin{figure}[ht!]
$$
\entrymodifiers={+[o]} \SelectTips{cm}{} \xymatrix{
&   &  \, \,    {   k (a) } \ar@{-}[d]^{2}      &   \\
&   &  \, \,    {   k(\Y_1) = k (T)  } \ar@{-}[dl]^{12}  \ar@{-}[d]^{18} \ar@{-}[dr]^{30}  &   \\
&  k(i_1)  &  \, \,    { k (i_2)  }     &  k(i_3) \\
}
$$
\label{fig_case2}
\end{figure}
One can compute equations defining $\Y_1$ as a subvariety of $\M_2$  (i.e.,  in terms of $i_1,i_2,i_3$) by
eliminating $T$. Such equations are large, hence we don't display them here. The degrees of this equation in $i_1,
i_2, i_3$ can be read from the above  diagram.
The elliptic curve $E$  has j-invariant
\begin{small}
\begin{equation*}
j = \frac {(11390625 T^8+1215000 T^7+99900 T^6 +925032 T^5+550 T^4+40 T^3+380 T^2-40 T+1)^3} {4096\, T^5(25
T-1)^2\, (25 T^2+6 T+1)^4\, (9 T-1)^6}
\end{equation*}
\end{small}
To compute $j$ we first write it as a rational function in $a$. Then, we make the substitution $T=\left(
\frac {a-2} {5 (a+2)} \right)^2$.

\begin{remark}
If $J_2=0$, then there are exactly six isomorphism classes of curves. In the moduli space they are described by
invariants $a_1, a_2$ as in \cite{Sh-V}.
\end{remark}

\subsubsection{Automorphism groups}
Next we want to find if the curves in $\Y_1$ have extra automorphisms.
\begin{prop}
Let $C$ be a genus 2 curve in the locus $\Y_1$. Then, the automorphism group of $C$ is either $\Z_2$ or $V_4$.
Moreover, there are exactly  36 isomorphism classes of curves  with automorphism group $V_4$ given by the
following values for $T$:

\begin{small}
\begin{equation*}
\begin{split}
(5625 T^3-650 T^2+73 T+8) (1265625 T^4-67500 T^3+89550 T^2+516 T+1) \\
(625 T^3-25 T^2-9 T+1)(5625 T^5+18075 T^4+8282 T^3+918 T^2-131 T-1) \\
(109375 T^5 +18125 T^4-12450T^3+1186 T^2-13 T+1)( 7119140625 T^7+ 6391406250 T^6\\
+2582859375 T^5+476007500 T^4+19626975 T^3-1411606 T^2+257473 T-4096)\\
(158203125000 T^9 +85869140625 T^8+32415625000 T^7+6116187500 T^6 +74885000 T^5\\
-94007050 T^4-7398504 T^3+1091468 T^2+48 T+1) = & \, \, 0
\end{split}
\end{equation*}
\end{small}
\end{prop}

\begin{proof}We substitute the expressions for  $i_1, i_2, i_3$ in the locus of curves with extra automorphisms given in
\cite{Sh-V}. Using methods developed in \cite{Sh-V}, it is an easy computational exercise to show that for each
$T$ as above, the automorphism group of the corresponding curve is the Klein 4-group. Moreover, the $i_1, i_2,
i_3$ as above don't satisfy any of the loci of curves with automorphism group $D_8$ or $D_{12}$. If $J_2=0$ then
we proceed similarly.

Since the above equation has degree 36, then there are at most 36 genus 2  curves corresponding to the above
values for $T$. We find the resultant with respect to $T$ of the above equation and the equation which gives $i_1$
in terms of $T$. We get a degree 36 polynomial in $i_1$ which has nonzero discriminant. Thus, we have 36 distinct
values for $i_1$ and therefore 36 distinct isomorphism classes of genus 2 curves.

\end{proof}

\subsection{Case II}

The condition $F_3(1)=0$ is equivalent to
$$b = a-1,$$
which we assume for the rest of this subsection. Then
$$z = \frac {(3a-1)(a-1)}{(2a+1)}$$
is the root of $F_4(x)$ not over $t=1$.  Thus, $C$ has  equation
$$y^2=x(x-1)(b_3 x^3 + b_2 x^2 + b_1 x + b_0 )$$
 where
\begin{equation}\label{b's}
\begin{split}
b_3 & = (2 a+1) (-8+9 a)^2 \\
b_2 & = -27 a^6-54 a^5+468 a^4-958 a^3+381 a^2+400 a-192\\
b_1 & = -18 a^6+380 a^5-1000 a^4+726 a^3+499 a^2-752 a+192 \\
b_0 & = (a+8)^2  (a-1)^3 (3a-1), \\
\end{split}
\end{equation}
where
$$a (9 a -8) (a+8) (2 a+1) (3 a-1) (a-1) (a-2)    \neq 0.$$
This condition is obtained by substituting $b=a-1$ in $\Delta (a, b)$. Summarizing we have the following:

\begin{prop}
Let $C$ be a genus 2 curve in the locus $\Y_2$. Then $C$  is given by
$$ y^2 = x (x-1) (b_3 x^3 + b_2 x^2 + b_1 x + b_0 )  $$
where $b_0, b_1, b_2, b_3$ are as in Eq.~\eqref{b's}. Moreover, the equation of $E$ is
$$ s^2= t\, ( t-1) \left(t - \frac {(3 a-1)^3 (a+8)^2 (a-1) } {27a (a-2)^5 }
\right)$$
\end{prop}

The proof is as for Case I).

\begin{remark} We see that if we fix the 4 branch points of $\phi$,
then there are 6 corresponding covers $\phi$. Again this is consistent with the fact that there are 6 classes of
generating systems of $S_5$ of cycle structure  $ (  (2)^2, (2)^2, (2)\cdot (3), (2) )$  and product 1.
\end{remark}

The absolute invariants $i_1, i_2, i_3$ of $C$  are rational functions in $$T:= \left( \frac a {a-2} \right)^2$$
of degree 8, 12, and 20 as follows:
\begin{small}
\begin{equation*}
\begin{split}
i_1 = & -  \frac 9 {J_2^2}\,
(1953125 T^8+100859375 T^7-133684375 T^6-17761750 T^5+60906155 T^4\\
& -14020705 T^3 + 115631 T^2 - 46816T - 256)\\ \\
i_2 = &  \frac {-27} { 8 J_2^3}\, (96435546875 T^{12}-4709765625000 T^{11}+10970742187500 T^{10}+4833343750000
T^9\\
&-31399133343750 T^8 +30923034102000 T^7 -13348926086820 T^6+3049853644080
T^5   \\
& -409782059325 T^4 +10407596440 T^3+1223394432 T^2 -18880512 T + 32768)\\
\\
i_3 = & - \frac {59049} {4096 J_2^5 }\,   T^2 (25 T-1)^5 (25 T-16)^5 (T-1)^7
\end{split}
\end{equation*}
\end{small}
where $J_2$ is
\begin{small}
\begin{equation*}
J_2 =  4375 T^4-12850 T^3+11457 T^2+458 T+43
\end{equation*}
\end{small}

\begin{prop} ${\mathbb C}(\Y_2) = {\mathbb C}(T)$
\end{prop}

\begin{proof}$T$ can be eliminated from the above expressions and expressed as a rational function in $i_1, i_2, i_3$.
\end{proof}

The elliptic curve $E$ associated with $C$  has j-invariant
\begin{small}
\begin{equation*}
\begin{split}
j = \frac {(9765625 T^6-23437500 T^5+19218750 T^4-6087500 T^3+560625 T^2+166368 T+256)^3} {729\, T (T-1)^2 (25
T-16)^4(25 T-1)^6}
\end{split}
\end{equation*}
\end{small}
To compute $j$ one proceeds similarly as in Case I).

\subsubsection{Automorphism groups}
\begin{prop}
Let $C$ be a genus 2 curve in the locus $\Y_2$. Then, the automorphism group of $C$ is either $\Z_2$ or $V_4$.
Moreover, there are exactly 25 isomorphism classes of curves  with automorphism group $V_4$ given by the following
values for $T$:

\begin{small}
\begin{equation*}
\begin{split}
(25 T^2+34 T+13) (25 T^2-26 T+10) (15625 T^3-3750 T^2-6075 T+32)\\
(225T^3-634 T^2-151 T-16) (625 T^4-800 T^3+156 T^2+74 T-1) \\
(625 T^5+22325 T^4+892131 T^3-338857T^2+48160 T-2304)  (421875 T^6+\\
2402500 T^5-6942350 T^4+5673748 T^3-1488397 T^2-20464 T-256) & =  \, \,  0
\end{split}
\end{equation*}
\end{small}
\end{prop}
The proof is computational and   similar to  the one in Case I.

\subsection{Case III)}
Let $C$ be a genus 2 curve admitting a degree 5 cover $ C \to E$ such that the corresponding cover $\phi:
\P^1 \to \P^1$  has  ramification structure
$$ (2)^2,\,  (2)^2,\,  (2)^2,\, (3)$$
Denote the locus of such curves in  $\M_2$ by $\Y_3$. This case is obtained when the ramified points over the last
two branch points coalesce, i.e.,  $F_4 (x)$ has a double root. We compute
\begin{equation}
Disc\,  (F_4 (x) ) \, = a\,( \, a^3+4 b a^2+4 a^2-12 b a+4 a b^2+4 a-16 b-16 b^2)
\end{equation}
If $a=0$ then  $\ph(x)=x$, a forbidden case. Let  $\bar \Y_3$ denote the  genus 1 curve
\begin{equation}\label{C3}
\Y_3: \qquad  a^3+4 b a^2+4 a^2-12 b a+4 a b^2+4 a-16 b-16 b^2 =0
\end{equation}
with $j$-invariant $j = \frac {702595369} {72900}$.

The equation of $C$ is given by
\begin{equation}\label{eq_C_3}
y^2 = x(x-1) \left( x- \frac {a^2+2 a b+2 a-2b } {2 (2a+1)}\right)
\left( x^2 - x - \frac {3 a (a^2-4)} {4 (2 a+1) (a-4) } \right)
\end{equation}
such that
\begin{small}
\begin{equation}
a (a^2-4)(2 a+1)(3 a^3-12 a-1)(a-4)(96 a^5-400 a^4-128 a^3+800 a^2-72 a-225) \neq 0
\end{equation}
\end{small}
The elliptic curve $E$ has equation
$$s^2= t (t-1) (t-\l)$$
where $\l = \ph (z)$ and
\[z  =  \frac {a^2+2 a b+2 a-2b } {2 (2a+1) }\]
We don't display $\l$ here but it is easily computable; see \cite{Eq}.

\medskip

\noindent   \textbf{Equations defining $\Y_3$:}   Let $u$ and $v$ be the $S_3$ - invariants. The condition
Eq.~\eqref{C3} is equivalent to
\begin{equation}\label{eq_u_v_case3}
2u+v-16=0.
\end{equation}
Denote this genus 0 curve by $\Y_3^\prime$. This implies that the discriminant of the field extension $\bC
(u, v, w)/ \bC(u, v)$
\[ \D_w =16(v-16+2u) (2u^3+u^2 v-36uv-16v^2-108v)(u-4v-2)^2(16v-4uv+u^2)^2v^2 \]
is $\D_w=0$. Hence, $[ \bC (u, v, w) : \bC (u, v)]$ =1 and
\begin{equation}\label{eq_y3}
\bC (u, v, w) = \bC (u, v).
\end{equation}
The absolute invariants $i_1, i_2, i_3$ of $C$ can now again be expressed in terms of $u$, $v$.  We have the
maps
\begin{equation}\label{phi_3}
\begin{split}
 \bar \Y_3 \overset{\Upsilon} \to & \Y_3^\prime \overset{\Psi } \to \M_2\\
 (a, b) \to &(u, v) \to (i_1, i_2, i_3)\\
\end{split}
\end{equation}
such that the image $\Psi  ( \Y_3^\prime) = \Y_3$. We know that $\deg  \ \Upsilon = 6$ and as we see next
$\deg \ \Psi =1$.

\begin{prop}The function field of $\Y_3$ is
\[ \bC (\Y_3) = \bC (u, v).\] Moreover, $u$ and $v$ can be computed explicitly in terms of the absolute
invariants $i_1, i_2, i_3$.
\end{prop}

\begin{proof}
The first part follows from Eq.~\eqref{eq_y3}. The second part is computational. Let $i_j = \frac {p_j(u, v)}
{q_j (u, v)}$ for $j=1, 2, 3$, where $p_i (u, v), q_i (u, v)$ are polynomials in $u, v$. Then, we have the
system
\begin{equation}\label{sys_3}
\left\{
\begin{split}
& 2u+v-16=0    \\
& i_1 \cdot q_1 (u, v)  - p_1 (u, v) =0      \\
& i_2 \cdot q_2 (u, v)  - p_2 (u, v) =0         \\
& i_3 \cdot q_3 (u, v)  - p_3 (u, v) =0          \\
\end{split}
\right.
\end{equation}
One can solve for  $u, v$ the system \eqref{sys_3} and get $u, v$ as rational functions in $i_1, i_2, i_3$.
\end{proof}

Eliminating $u$, $u$ from the system \eqref{sys_3} yields equations defining $\Y_3$ as a subvariety of
$\M_2$. However, such equations are large and we don't display them. For example, $\Y_3$ is given as
$$G( i_1, i_2) =0,$$
where $G( i_1, i_2)$ is a polynomial in $i_1, i_2$ of degree 96 and 64 respectively. An equation of $\Y_3$ in
terms of $i_2, i_3$ has degrees 160 and 96 respectively; see \cite{Eq} for details.

\begin{remark}
The reader who is interested in such equations can use any computational algebra package such as Maple, Magma etc
and eliminate $u, v$ via resultants. We used \textsc{Maple 9} and computations sometimes took several days.
\end{remark}

\subsubsection{Automorphism groups}
\begin{prop}
Let $C$ be a genus 2 curve in the locus $\Y_3$. Then the automorphism group of $C$ is either $\Z_2$ or $V_4$.
Moreover, there are 138   isomorphism classes of curves  with automorphism group $V_4$ given by the following
values for $a$:

\begin{small}
\begin{equation}
\begin{split}\label{aut_3}
(24 a^5-84 a^4-144 a^3+328 a^2+220 a+17) (96 a^5-400 a^4-131 a^3+800 a^2-60 a-224)^2\\
(72 a^5-316 a^4+16 a^3+472 a^2-292 a-241)^2 (9 a^6-72 a^4-6 a^3+152 a^2-4 a-15) \\
(9216 a^{10}-76800 a^9+136000 a^8+252832 a^7-634615 a^6-184640 a^5+824616a^4-57222 a^3\\
-340360 a^2+30348 a+47025) (216 a^{11}+1548 a^{10}-20688 a^9+25776a^8+133824 a^7\\
-190976 a^6-286296 a^5+289508 a^4+231440 a^3-65056 a^2-58032 a-6975) (27648a^{13}\\
-230400 a^{12}+295680 a^{11}+1689600 a^{10}-3531264 a^9-3711808 a^8+10386272 a^7\\
+2095872a^6-11895424 a^5+1027312 a^4+5156035 a^3-398800 a^2-849036 a-61696) \\
(648 a^{14}+4644a^{13}-64656 a^{12}+58320 a^{11}+658152 a^{10}-935328 a^9-2364128 a^8\\
+3266608 a^7  +3718976a^6-3536792 a^5 -2448532 a^4+439027 a^3+404320 a^2\\
+174132 a+57600) = 0 &
\end{split}
\end{equation}
\end{small}

\end{prop}

\begin{proof}We substitute the expressions for  $i_1, i_2, i_3$ in the locus of curves with extra automorphisms given in
\cite{Sh-V}. Thus, all genus 2 curves obtained by the above values of $a$ have $V_4$ embedded in their
automorphism group. Moreover,  for each $a$ as above, the automorphism group of the corresponding curve is not
$D_8$, $D_{12}$. Hence, it is  Klein 4-group.  If $J_2=0$ then we proceed similarly.

The invariant  $i_1 = \frac {p(u, v) } {q(u, v) }$, where $p(u, v) , \, q(u, v) $ are polynomials in $u, v$.
We take the resultant of $i_1 \, q(u, v)  - p(u, v) $ and $Disc \, F_4(x)$ with respect to $b$. The result is
a polynomial in $i_1$ and $a$ of degree 2 and 69 respectively. Taking the resultant of this polynomial and
the polynomial of degree 69 in Eq.~\eqref{aut_3} we get a polynomial in $i_1$ of degree 138. The discriminant
of this polynomial is nonzero. Hence, there are 138 isomorphism classes of genus 2 curves. This is to be
expected since for each $a$ we have two values of $b$ determined from Eq.~\eqref{C3}.

\end{proof}
\section{Concluding remarks}

The main goal of this paper is to compute an equation for the locus $\L_5$ and its subloci. This equation is in
terms of the Igusa coordinates $i_1$, $i_2$, $i_3$ on the moduli space of genus 2 curves.  That $\L_5$ is a
rational variety follows also from the general theory of "diagonal modular surfaces", see Kani \cite{Kani}. The
computations performed give us precise information on the locus $\L_5$ and its degenerate subloci which would be
difficult to obtain with other methods. Since the equations describing $\L_5$ and its subloci are big and  take
several pages to display we chose not to display them. Instead, we only gave birational parametrizations of these
spaces.

For the reader who wants to use such equations  but doesn't want to go through the lengthy computations of
obtaining them we can provide them; see \cite{Eq}.

Degree $n$ covers $\psi: C \to E$  have been successfully used in number theory applications. The genus 2
curve with the largest known number of rational points has automorphism group isomorphic to $D_{12}$; thus it
has degree 2 cover to an elliptic curve. It was found by Keller and Kulesz and it is known to have at least
588 rational points; see \cite{KK}. Using degree $n=2,3$  covers $\psi: C \to E$ Howe, Leprevost, and Poonen
\cite{HLP} were able to construct a family of  genus 2 curves whose Jacobians each have large  rational
torsion subgroups. Using formulas in Theorem 1 and 2, similar techniques probably could be applied using
degree $n \geq 5$ covers.

The moduli spaces $\L_n$ also have applications to the study of integrable systems. The interested reader
should check \cite{AP} for a complete survey on this topic.   The authors describe the moduli space of genus
two curves that admit a degree $n$ elliptic subcover in several ways: by algebra, group theory, monodromy,
and topology.   We hope this paper fills the gap in the literature of explicitly describing such moduli
spaces as subspaces of the moduli space $\M_2$.


\end{document}